\documentclass[12pt, reqno, a4paper]{amsart}
\usepackage{latexsym, fancyheadings, amsmath, amssymb, amsthm, amsbsy, graphicx}
\usepackage[american]{babel}
\usepackage{hyperref}

\theoremstyle{definition}
\newtheorem{theorem}{Theorem}
\newtheorem{defin}[theorem]{Definition}
\newtheorem{remark}[theorem]{Remark}
\newtheorem{corollary}[theorem]{Corollary}
\newtheorem{lemma}[theorem]{Lemma}

\begin{document}

\newcommand{\FirstLeftHeader}[1]{\lhead{\fancyplain{\footnotesize #1}{}}}
\newcommand{\Headers}[2]{\chead[\fancyplain{}{\footnotesize #1}]{\fancyplain{}{\footnotesize #2}}}
\newcommand{\FirstRightHeader}[1]{\rhead{\fancyplain{\footnotesize #1} {}}}
\newcommand{\AuthorInfo}[1]{\bigskip \noindent {\sc #1}\bigskip}
\newcommand{\Caption}[2]{\medskip\begin{center}{\sc #1} #2\end{center}\smallskip}

\makeatletter
\renewcommand{\labelenumi}{\theenumi.}
\renewcommand{\labelitemi}{ -- }
\renewcommand{\@biblabel}[1]{#1.\hfill}
\renewcommand{\refname}{\sc Bibliography}

\newenvironment{Abstract} { \vspace{20pt}
 \begin{quote}\footnotesize {\sc Abstract.}}
 {\end{quote}\vspace{10pt}}

\newenvironment{Udc}{\begin{flushleft} UDC}
 {\end{flushleft}}
\newenvironment{Title}{\vspace{10pt} \bf \begin{center}}
 {\end{center}}
\newenvironment{Authors}{\vspace{20pt} \sc \begin{center}}
 {\end{center}}
\newenvironment{Head}{\vspace{10pt} \sc \begin{center}}
 {\end{center}}

\numberwithin{equation}{section}

\pagestyle{fancyplain}
\thispagestyle{plain}

\Headers{\textsc{D.\,Ya. Khusainov, M. Pokojovy, E. I. Azizbayov}}{\textsc{Classical Solvability for a Linear Heat Equation with Constant Delay}}

\begin{Title}
	On Classical Solvability for a Linear 1D Heat Equation with Constant Delay
\end{Title}
\begin{Authors}
	{D. Ya. Khusainov, M. Pokojovy, E. I. Azizbayov}
\end{Authors}
\vspace{0.25cm}
\begin{center}
	August 1st, 2013
\end{center}
\begin{Abstract}
	In this paper, we consider a linear heat equation with constant coefficients and a single constant delay.
	Such equations are commonly used to model and study various problems
	arising in ecology and population biology
	when describing the temporal evolution
	of human or animal populations accounting for migration, interaction with the environment and certain aftereffects
	caused by diseases or enviromental polution, etc.
	(see \cite{Ba1998}, \cite{OkuLe2001} and references therein).
	Whereas dynamical systems with lumped parameters have been addressed in numerous investigations (cf. \cite{Go1992}, \cite{Ha1977}),
	there are still a lot of open questions for the case of systems with distributed parameters (see, e.g., \cite{LaTri2010}, \cite{LaTri2011}),
	especially when the delay effects are incorporated (cp. \cite{BaPia2005}, \cite{BaSchn2004}).

	The present paper is an elaboration of authors' results in \cite{AzKhu2012}.
	Here, we consider a general non-homogeneous one-dimensional heat equation with delay in both higher and lower order terms
	subject to non-homogeneous initial and boundary conditions.
	For this, we prove the unique existence of a classical solution
	as well as its continuous dependence on the data.
	\vspace{1mm}
\end{Abstract}

\section{Linear Heat Equation Without Delay}
We consider an initial boundary value problem for a one-dimensional heat equation without delay
\begin{equation}
	v_t(x, t) = a^2 v_{xx} (x, t) + b v_x (x, t) + c v(x, t) + g(x, t) \text{ for } x \in (0, l), t > 0 \label{EQ1}
\end{equation}
subject to non-homogeneous Dirichlet boundary conditions
\begin{equation}
	v(0, t) = \theta_1 (t), \quad v(l, t) = \theta_2 (t) \text{ for } t > 0 \label{EQ2}
\end{equation}
as well as initial conditions
\begin{equation}
	v(x, 0) = \psi(x) \text{ for } x \in (0, l). \label{EQ3}
\end{equation}
Since we are interested in classical solutions,
compatibility conditions on the initial and boundary data are additionally posed
\begin{equation}
	\psi(0) = \theta_1(0), \quad \psi(l) = \theta_2(0) \notag
\end{equation}
allowing for the continuity of the solution at the boundary of the space-time cylinder.
\begin{defin}
	Under a classical solution of the problem (\ref{EQ1})--(\ref{EQ3}) on a finite time interval $[0, T]$
	we understand a function $v \in \mathcal{C}^{0}\big([0, l] \times [0, T]\big)$ which
	satisfies $v_{t}, v_{xx} \in \mathcal{C}^{0}\big([0, l] \times [0, T]\big)$
	and, being plugged into Equations (\ref{EQ1})--(\ref{EQ3}), turns them into identity.
\end{defin}

The uniqueness of solutions can be deduced from the weak maximum principle (cf. \cite[p. 117]{ReRo2004}).
Here, we decided for a proof based on the energy method (cp. \cite[Bd. 2, Kap. 23]{DeRa2011})
from which we can also conclude the continuous dependence of the solution on the data.
\begin{theorem}
	For each $T > 0$, classical solutions on $[0, T]$ are unique.
\end{theorem}

\begin{proof}
	We assume that there exist two classical solution $v_{1}$, $v_{2}$ to the initial boundary value problem (\ref{EQ1})--(\ref{EQ3}).
	Then their difference $w := v_{1} - v_{2}$ is a classical solution to the homogeneous initial boundary value problem
	\begin{equation}
		\begin{split}
			w_{t}(x, t) &= a^2 w_{xx} (x, t) + b w_x (x, t) + c w(x, t) \text{ for } (x, t) \in (0, l) \times (0, T), \\
			w(0, t) &= w(l, t) = 0 \text{ for } t \in (0, T), \\
			w(x, 0) &= 0 \text{ for } x \in (0, l).
		\end{split}
		\notag
	\end{equation}
	Multiplying the equation with $w$, integrating over $x \in (0, l)$ and applying Green's formula,
	we obtain using the theorem on differentiation under the integral sign
	\begin{equation}
		\begin{split}
			\frac{1}{2} \partial_{t} \int_{0}^{l} w^{2}(x, t) \mathrm{d}x &=
			-a^{2} \int_{0}^{l} w_{x}^{2}(x, t) \mathrm{d}x + b \int_{0}^{l} w_{x}(x, t) w(x, t) \mathrm{d}x + \\
			&\phantom{=}\;\; \phantom{-} c \int_{0}^{l} w^{2}(x, t) \mathrm{d}x.
		\end{split}
		\notag
	\end{equation}
	Exploiting Young's inequality
	\begin{equation}
		|\xi \eta| \leq \frac{\varepsilon}{2} \xi^{2} + \frac{1}{2 \varepsilon} \eta^{2} \notag
	\end{equation}
	for $\xi, \eta \in \mathbb{R}$, $\varepsilon > 0$,
	we can further estimate
	\begin{equation}
		\frac{1}{2} \partial_{t} \int_{0}^{l} w^{2}(x, t) \leq
		-\Big(a^{2} - \varepsilon \frac{|b|}{2}\Big) \int_{0}^{l} w_{x}^{2}(x, t) \mathrm{d}x
		+ \Big(c + \frac{|b|}{2 \varepsilon}\Big) \int_{0}^{l} w^{2}(x, t) \mathrm{d}x.
		\notag
	\end{equation}
	Letting now $\varepsilon$ be sufficiently small such that $\varepsilon \frac{|b|}{2} < a^{2}$,  we obtain
	\begin{equation}
		\partial_{t} \int_{0}^{l} w^{2}(x, t) \leq C \int_{0}^{l} w^{2}(x, t) \mathrm{d}x \notag
	\end{equation}
	with $C := 2 \big(c + \frac{|b|}{2 \varepsilon}\big)$.
	As an immediate consequence of Gronwall's inequality, we get then
	\begin{equation}
		\int_{0}^{l} w^{2}(x, t) \mathrm{d}x \leq
		e^{Ct} \int_{0}^{l} w^{2}(x, 0) \mathrm{d}x = 0 \text{ for a.e. } t \in [0, T]. \notag
	\end{equation}
	Taking into account the continuity of $w$, we finally get $w \equiv 0$ and therefore $v_{1} \equiv v_{2}$.
\end{proof}

After a slight modification of the proof, we easily obtain the continuous dependence of the solution on the data.
See \cite{AdFou2003} for the definition of corresponding Sobolev spaces.
\begin{corollary}
    The norm of the solution $v$ in $L^{2}\big((0, T), L^{2}\big((0, l)\big)\big)$
    depends continuously on the
    $L^{2}\big((0, T), L^{2}\big((0, l)\big)\big) \times L^{2}\big((0, l)) \times \left(W^{1, 2}\big((0, T)\big)\right)^{2}$-norm of $(g, \psi, \theta_{1}, \theta_{2})$.
\end{corollary}

\begin{remark}
	For the sake of consistency with the traditional convention used for the spaces of Banach-valued functions,
	here and in the sequel we interchange the $x$ and $t$ variables when dealing with functions in Sobolev or Lebesgue spaces.
	Thus, we write $u = u(t, x) \in H^{1}\big((0, T), L^{2}\big((0, l))\big)$,
	but $v = v(x, t) \in \mathcal{C}^{0}\big([0, l] \times [0, T]\big)$.
\end{remark}

Now we want to establish the existence of classical solutions and give their explicit representation.
First, we substitute
\begin{equation}
	v(x, t) := e^{\mu x + \gamma t} u(x, t), \quad \mu := -\frac{b}{2a^{2}}, \quad \gamma := c - \Big(\frac{b}{2a}\Big)^{2} \notag
\end{equation}
and find an equivalent system for the function $u$ given by
\begin{equation}
	u_{t}(x, t) = a^{2} u_{xx}(x, t) + f(x, t) \text{ for } x \in (0, l), t > 0 \label{NEW_EQ_1}
\end{equation}
with $f(x, t) := e^{-\mu x - \gamma t} g(x, t)$ subject to the initial and boundary conditions
\begin{align}
	u(x, 0) &= \varphi(x) \text{ for } x \in (0, l), \quad \varphi(x) := e^{-\mu x} \psi(x), \label{NEW_EQ_2} \\
	u(0, t) &= \mu_{1}(t) := e^{-\gamma t} \theta_{1}(t), \;
	u(l, t) = \mu_{2}(t) := e^{-\mu l - \gamma t} \theta_{2}(t) \text{ for } t > 0. \label{NEW_EQ_3}
\end{align}
Thus, the problem (\ref{EQ1})--(\ref{EQ3}) is reduced to the problem (\ref{NEW_EQ_1})--(\ref{NEW_EQ_3}).
We look for the classical solution $u$ to (\ref{NEW_EQ_1})--(\ref{NEW_EQ_3}) in the form
\begin{equation}
	u(x,t) = u_1 (x,t) + u_2 (x,t) + u_3 (x,t),
	\notag
\end{equation}
where
\begin{itemize}
	\item $u_1$ is the solution to the homogeneous parabolic equation
	\begin{equation}
		\frac{{\partial u_1 (x,t)}}{{\partial t}} = a^2 \frac{{\partial ^2 u_1 (x,t)}}{{\partial x^2}}
		\label{EQ4}
	\end{equation}
	with zero boundary conditions $u_1 (0,t) = 0$, $u_1 (l,t) = 0$, $t > 0$,
	and non-zero initial conditions $u_1(x, 0) = \Phi(x)$, $x \in (0, l)$,
	where
	\begin{equation}
		\Phi(x) := \varphi(x) - \mu_1 (0) - \frac{x}{l}\left[{\mu _2 (0) - \mu _1 (0)} \right] \text{ for } x \in [0, l].
		\label{EQ5}
	\end{equation}

	\item $u_2$ is the solution to the non-homogeneous parabolic equation
	\begin{equation}
		\frac{{\partial u_2 (x,t)}}{{\partial t}} = a^2 \frac{{\partial ^2 u_2 (x,t)}}{{\partial x^2 }} + F(x,t)
		\label{EQ6}
	\end{equation}
	with the right-hand side
	\begin{equation}
		\begin{split}
			F(x,t) := &f(x,t) - \frac{\mathrm{d}}{{\mathrm{d}t}}\left\{ {\mu _1 (t) + \frac{x}{l}\left[ {\mu _2 (t) - \mu _1 (t)} \right]} \right\} + \\
			&c\left\{ {\mu _1 (t) + \frac{x}{l}\left[ {\mu _2 (t) - \mu _1 (t)} \right]} \right\}
		\end{split}
		\label{EQ7}
	\end{equation}
	for $x \in [0, l]$, $t \geq 0$,
	subject to zero boundary conditions $u_2(0, t) = 0$, $u_2 (l,t) = 0$, $t > 0$
	and zero initial conditions $u_2(x, 0) = 0$, $0 < x < l$.

	\item $u_3$ is the solution to the family of elliptic equations
	\begin{equation}
		a^{2} \frac{{\partial^2 u_3 (x,t)}}{{\partial x^2}} = 0 \notag
	\end{equation}
	subject to non-zero boundary conditions
	$u_3(0, t) = \mu_{1}(t)$, $u_3(l, t) = \mu_{2}(t)$, $t > 0$.
	Thus, $u_{3}(x,t) = \mu _1 (t) + \frac{x}{l}\left[{\mu _2 (t) - \mu _1 (t)}\right].$
\end{itemize}

\subsection{Homogeneous Equation}
\label{SECTION_HOM_EQUATION_NO_DELAY}
We first study the homogeneous Equation (\ref{EQ4}) with the initial conditions given in Equation (\ref{EQ5}).
Using Fourier's separation method, the solution is to be determined in the form
\begin{equation}
	u_1(x, t) = X(x) T(t).
	\notag
\end{equation}
Plugging this ansatz into (\ref{EQ4}), we arrive at
\begin{equation}
	X(x)T'(t) = a^2 X''(x)T(t). \notag
\end{equation}
Collecting corresponding terms, we get
\begin{equation}
	T'(t) X(x) = a^2 X''(x) T(t).
	\notag
\end{equation}
After separating the variables
\begin{equation}
	\frac{X''(x)}{X(x)} = \frac{{T'(t)}}{a^{2} T(t)} =  -\lambda^{2},
	\notag
\end{equation}
the equation decomposes into two equations
\begin{align}
	X''(x) + \lambda^{2} X(x) &= 0, \label{EQ8} \\
	T'(t) &=  -a^{2} T(t). \label{EQ9}
\end{align}
Using the boundary conditions for $u_{1}$, we obtain zero boundary conditions for $X$
\begin{equation}
	X(0) = 0, \quad X(l) = 0.
	\notag
\end{equation}
Nontrivial solutions of Equation (\ref{EQ8}) exist only for the eigenvalues
\begin{equation}
	\lambda^{2} = \lambda_n^{2} = \Big(\frac{\pi n}{l}\Big)^{2}, \quad n \in \mathbb{N} \notag
\end{equation}
with corresponding eigenfunctions
\begin{equation}
	X_n(x) = \sin \frac{\pi n}{l} x, \quad n \in \mathbb{N},
	\label{EQ11}
\end{equation}
being solutions of the Sturm \& Liouville problem for the negative Dirichlet-Laplacian in $(0, l)$ (cp. also Definition \ref{DEFINITION_OPERATOR} below).
Note that $(X_{n})_{n \in N}$ build an orthogonal basis of $L^{2}\big((0, l)\big)$.

Plugging the values of $\lambda_n$, $n \in \mathbb{N}$, obtained above into Equation (\ref{EQ9}),
we get a countable system of decoupled ordinary differential equations
\begin{equation}
	T'(t) = -\left(\frac{\pi n}{l} a\right)^{2} T(t), \quad n \in \mathbb{N},
	\label{EQ12}
\end{equation}
which is uniquely solved by the sequence of analytic functions
\begin{equation}
	T_n(t) = T(0)e^{-\left(\frac{\pi n}{l} a\right)^{2} t}, \quad
	t \geq 0, \quad n \in \mathbb{N}.
	\notag
\end{equation}
Since $(X_{n})_{n \in \mathbb{N}}$ build an orthogonal basis of $L^{2}\big((0, l)\big)$
(cf. \cite[Theorem 9.22]{ReRo2004}),
the function $\Phi$ can be expanded into a Fourier series with respect to the eigenfunction from Equation (\ref{EQ11}), viz.,
\begin{equation}
	\Phi (x) = \sum\limits_{k = 1}^\infty \Phi_n \sin \frac{{\pi n}}{l}x \text{ for a.e. } x \in [0, l]
	\notag
\end{equation}
with
\begin{equation}
	\Phi_n  = \frac{2}{l}\int\limits_0^l {\left\{ {\varphi (\xi ) - \left[ {\mu _1 (0) + \frac{\xi }{l}\left[ {\mu _2 (0) - \mu _1 (0)} \right]} \right]} \right\} \sin \frac{{\pi n}}{l}\xi} \mathrm{d}\xi.
	\notag
\end{equation}
\begin{defin}
	\label{DEFINITION_OPERATOR}
	Consider the elliptic operator $\mathcal{A} := -a^{2} \partial_{x}^{2}$ on $L^{2}\big((0, l)\big)$ subject to homogeneous Dirichlet boundary conditions.
	Since $\mathcal{A}$ is continuoulsy invertible, $0 \in \rho(\mathcal{A})$.
	For $m \in \mathbb{N}$, we define the space
	\begin{equation}
		X_{m} := D(\mathcal{A}^{m}) = \left\{u \in H^{2m}\big((0, l)\big) \,\Big|\, \partial_{x}^{2k} u \in H^{1}_{0}\big((0, l)\big), k = 0, \dots, m-1\right\}
		\notag
	\end{equation}
	equipped with the standard graph norm of $D(\mathcal{A}^{m})$.
\end{defin}

\begin{remark}
	By the virtue of elliptic theory (cf. \cite{ReRo2004}),
	$X_{m}$ is well-defined and
	the norm of $X_{m}$ is equivalent with the standard norm of $H^{2m}\big((0, l)\big)$.
\end{remark}

\begin{lemma}
	\label{LEMMA_FOURIER_DECAY}
	For any $m \in \mathbb{N}$ and any $w \in X_{m}$, there exists a constant $C > 0$ such that
	the Fourier coefficients $w_{n}$, $n \in \mathbb{N}$, of $w$ with respect to the functions basis $(X_{n})_{n \in \mathbb{N}}$ given in Equation (\ref{EQ11}) satisfy
	\begin{equation}
		|w_{n}| \leq \frac{C}{n^{2m + 1/2}} \text{ for } n \in \mathbb{N}. \notag
	\end{equation}
\end{lemma}
\begin{proof}
	Again, from the elliptic theory, we know that $w \in X_{m}$ is equivalent with
	\begin{equation}
		\sum_{n = 1}^{\infty} \big(n^{2}\big)^{2m} |w_{n}|^{2} < \infty, \notag
	\end{equation}
	where $w_{n}$ denotes the $n$-th Fourier coefficient of $w$ with respect to $(X_{n})_{n}$.
	Thus, there exists a constant $C > 0$ such that
	$n^{4m} |w_{n}|^{2} \leq \frac{C^{2}}{n}$ and therefore $|w_{n}| \leq \frac{C}{n^{2m + 1/2}}$.
\end{proof}

The Fourier series converges in $L^{2}\big((0, l)\big)$ if $\Phi \in L^{2}\big((0, l)\big)$.
For $\Phi \in \mathcal{C}^{0}\big([0, l]\big)$ with $\Phi(0) = \Phi(l) = 0$, the convergence is pointwise (cf. \cite[Bd. 1, Kap. 9]{DeRa2011}).
Under a stronger condition, e.g., $\Phi \in X_{1} \hookrightarrow W^{1, \infty}\big((0, l)\big)$, the convergence is even uniform.
See \cite{AdFou2003} for the definition of corresponding Sobolev spaces.

Then, the solution to the initial boundary value problem (\ref{EQ4})--(\ref{EQ5}) is formally given by
\begin{equation}
	u_1(x,t) = \sum\limits_{n = 1}^\infty \Phi_n e^{-\left(\frac{\pi n}{l} a\right)^{2} t} \sin \frac{\pi n}{l} x.
	\label{EQ13}
\end{equation}
Assuming $\Phi \in X_{2}$, we easily conclude from Lemma \ref{LEMMA_FOURIER_DECAY}
that $u_{1}$ given Equation (\ref{EQ13}) as well as $\partial_{t} u_{1}$, $\partial_{xx} u_{1}$ converge absolutely and uniformly on $[0, l] \times [0, T]$.
Thus, $u_{1}$ is a classical solution of (\ref{EQ4})--(\ref{EQ5}).

\subsection{Non-Homogeneous Equation}
\label{SECTION_NON_HOM_EQUATION_NO_DELAY}
Next, we consider the non-homogeneous equation (\ref{EQ6})
\begin{equation}
	\frac{{\partial u_2 (x,t)}}{{\partial t}} = a^2 \frac{{\partial^2 u_2 (x,t)}}{{\partial x^2 }} + F(x,t)
	\notag
\end{equation}
subject to zero boundary conditions $u_2(0, t) = 0$, $u_2(l, t) = 0$, $t > 0$,
and zero initial conditions $u_2(x, 0) = 0$, $0 < x < l$.
Using Duhamel's principle, the solution will be determined
as a Fourier series with time-dependent coefficients with respect to the eigenfunctions $(X_{n})_{n \in \mathbb{N}}$, i.e.,
\begin{equation}
	u_2(x,t) = \sum\limits_{n = 1}^\infty u_{2n} (t) \sin \frac{{\pi n}}{l}x, \quad n \in \mathbb{N}.
	\label{EQ14}
\end{equation}
Note that $(X_{n})_{n \in \mathbb{N}}$ defined in the previous subsection can be extended to an orthogonal basis of $L^{2}\big((0, T), L^{2}\big((0, l)\big)\big)$.
Thus, if the right-hand side of Equation (\ref{EQ7}) satisfies
$F \in L^{2}\big((0, T), L^{2}\big((0, l)\big)\big)$,
it can be expanded into a Fourier series with respect to this function basis.
We represent the function $F$ in the form of the series
\begin{equation}
	F(x,t) = \sum\limits_{n = 1}^\infty  F_n(t) \sin \frac{\pi n}{l}x, \quad
	F_n (t) = \frac{2}{l} \int\limits_0^l F(\xi, t) \sin \frac{{\pi n}}{l}\xi \mathrm{d}\xi,
	\notag
\end{equation}
where $F_{n} \in L^{2}\big((0, T)\big)$, $n \in \mathbb{N}$.
With $u_{2, n} \in H^{1}\big((0, T)\big)$, $n \in \mathbb{N}$,
\begin{equation}
	u_{2, n}(t) = \int\limits_0^t {e^{-\left(\frac{\pi n}{l}\right)^{2} (t - s)} F_n (s)\mathrm{d}s}
	\notag
\end{equation}
representing the unique solution to the ordinary differential equation
\begin{equation}
	\dot u_{2, n} (t) = -\left(\frac{\pi n}{l}\right)^{2} u_{2n} (t) + F_n (t)
	\notag
\end{equation}
subject to zero initial condition $u_{2, n}(t) = 0$, the solution $u_{2}$ to Equation (\ref{EQ6})--(\ref{EQ7}) is formally given by
\begin{equation}
	u_2 (x,t) = \sum\limits_{n = 1}^\infty  {\left[ {\int\limits_0^t {e^{-\left(\frac{\pi n}{l}\right)^{2} (t - s)} } F_n (s)\mathrm{d}s} \right]} \sin \frac{{\pi n}}{l}x.
	\label{EQ15}
\end{equation}

From Lemma \ref{LEMMA_FOURIER_DECAY}, we infer that both the series $u_{2}$ given in Equation (\ref{EQ15}) 
and its partial derivatives $\partial_{t} u_{2}$, $\partial_{xx} u_{2}$ converge absolutely and uniformly
if, e.g., $F \in \mathcal{C}^{1}\big([0, T], X_{1}\big) \cap \mathcal{C}^{0}\big([0, T], X_{2}\big)$.
Thus, $u_{2}$ is a classical solution of the corresponding problem.

\subsection{Elliptic Equation}
Trivially, we observe that
$u_{3} \in \mathcal{C}^{2}\big([0, T], \mathcal{C}^{\infty}\big([0, l]\big)\big)$
if $\mu_{1}, \mu_{2} \in \mathcal{C}^{2}\big([0, T]\big)$.
Summarizing the relations obtained above, we arrive at
\begin{align}
	u(x,t) &= \sum\limits_{n = 1}^\infty \Phi_n e^{-\left(\frac{\pi n}{l}\right)^{2} t}\sin \frac{{\pi n}}{l}x  + \label{EQ16} \\
	&\phantom{=}\;\; \sum\limits_{n = 1}^\infty \left[\int\limits_0^t e^{-\left(\frac{\pi n}{l}\right)^{2} (t - s)} F_n (s)\mathrm{d}s \right] \sin \frac{{\pi n}}{l}x + \mu _1 (t) + \frac{x}{l}\left[ {\mu _2 (t) - \mu _1 (t)} \right] \notag
\end{align}
with
\begin{equation}
	\begin{split}
		\Phi_n &= \frac{2}{l}\int\limits_0^l \left\{ {\varphi (\xi ) - \left[ {\mu _1 (0) + \frac{\xi }{l}\left[ {\mu _2 (0) - \mu _1 (0)} \right]} \right]} \right\} \sin \frac{{\pi n}}{l}\xi \mathrm{d} \xi, \\
		F_n(t) &= \frac{2}{l}\int\limits_0^l {F(\xi ,t)  \sin \frac{{\pi n}}{l}\xi \mathrm{d}\xi }, \\
		F(x,t) &= f(x,t) - \frac{\mathrm{d}}{{\mathrm{d}t}}\left\{ {\mu _1 (t) + \frac{x}{l}\left[ {\mu _2 (t) - \mu _1 (t)} \right]} \right\}.
	\end{split}
	\notag
\end{equation}

\begin{theorem}
	Assume
	\begin{equation}
		\begin{split}
			f &\in \mathcal{C}^{1}\big([0, T], H^{2}\big((0, l)\big)\big) \cap \mathcal{C}^{0}\big([0, T], H^{4}\big((0, l)\big)\big), \\
			\mu_{1}, \mu_{2} &\in \mathcal{C}^{2}\big([0, T]\big), \quad
			\varphi \in H^{4}\big((0, l)\big)
		\end{split}
		\notag
	\end{equation}
	as well as the compatibility conditions
	\begin{equation}
		\varphi(0) = \mu_{1}(0), \quad \varphi(l) = \mu_{2}(0) \notag
	\end{equation}
	and
	\begin{equation}
		f(0, t) = \dot{\mu}_{1}(t), \quad
		f(l, t) = \dot{\mu}_{2}(t), \quad
		f_{xx}(0, t) = 0, \quad
		f_{xx}(l, t) = 0
		\text{ for } t \in [0, T].
		\notag
	\end{equation}
	Then the function $u$ given in Equation (\ref{EQ16}) is a unique classical solution to (\ref{EQ1})--(\ref{EQ3}).
\end{theorem}

\section{Linear Heat Equation with Delay}
In this section, we consider a linear one-dimensional heat equation with constant coefficients
and a single constant delay
\begin{equation}
	\begin{split}
		v_t (x,t) &= a_1^2 v_{xx} (x,t) + a_2^2 v_{xx} (x,t - \tau ) + b_1 v_x (x,t) + b_2 v_x (x,t - \tau ) + \\
		&\phantom{=}\;\; d_1 v(x,t) + d_2 v(x,t - \tau ) + g(x,t) \text{ for } x \in (0, l), \quad t > 0.
	\end{split}
	\label{EQ2_1}
\end{equation}
with $a_{1}, a_{2} \neq 0$.
Equation (\ref{EQ2_1}) is complemented by non-homogeneous Dirichlet boundary conditions
\begin{equation}
	v(0, t) = \theta_1(t), \quad u(l,t) = \theta_2 (t) \text{ for } t > -\tau \label{EQ2_2}
\end{equation}
and initial conditions
\begin{equation}
	v(x, t) = \psi(x, t) \text{ for } x \in (0, l), \quad t \in (-\tau, 0). \label{EQ2_3}
\end{equation}
Since we are again interested in classical solutions,
the following compatibility conditions are going to be essential to assure
the continuity of the solution on the boundary of the space-time cylinder
\begin{equation}
	\psi(0, t) = \theta_1(t), \quad 
	\psi(l, t) = \theta_2 (t) \text{ for } t \in [-\tau, 0].
	\notag
\end{equation}

\begin{defin}
	A function
	$v \in \mathcal{C}^{0}\big([0, l] \times [-\tau, T]\big)$ satisfying
	$\partial_{t} v \in \mathcal{C}^{0}\big([0, l] \times [0, T]\big)$,
	$\partial_{xx} v \in \mathcal{C}^{0}\big([0, l] \times [0, T]\big) \cap \mathcal{C}^{0}\big([0, l] \times [-\tau, 0]\big)$
	is called a classical solution to the problem (\ref{EQ2_1})--(\ref{EQ2_3}) on a finite time interval $[0, T]$
	if it, being plugged into Equations (\ref{EQ2_1})--(\ref{EQ2_3}), turns them into identity.
\end{defin}

\begin{theorem}
	For any $T > 0$, classical solutions of the problem (\ref{EQ2_1})--(\ref{EQ2_3}) on $[0, T]$ are unique.
\end{theorem}

\begin{proof}
	We assume that there exist two classical solution $v_{1}$, $v_{2}$ to the initial boundary value problem with delay (\ref{EQ2_1})--(\ref{EQ2_3}).
	Then their difference $w := v_{1} - v_{2}$ is a classical solution to the homogeneous problem
	\begin{equation}
		\begin{split}
			w_t (x,t) &= a_1^2 w_{xx} (x,t) + a_2^2 w_{xx} (x,t - \tau ) + b_1 w_x (x,t) + b_2 w_x (x,t - \tau) + \\
			&\phantom{=}\;\; d_1 w(x,t) + d_2 w(x,t - \tau) \text{ for } (x, t) \in (0, l) \times (0, T), \\
			w(0, t) &= w(l, t) = 0 \text{ for } t \in (-\tau, T), \\
			w(x, t) &= 0 \text{ for } (x, t) \in (0, l) \times (-\tau, 0).
		\end{split}
		\notag
	\end{equation}
	Multiplying the equation with $w$, integrating over $x \in (0, l)$ and applying Green's formula,
	we obtain using the theorem on differentiation under the integral sign
	\begin{align}
		\frac{1}{2} \partial_{t} \int_{0}^{l} w^{2}(x, t) \mathrm{d}x &=
		\int_{0}^{l} \left(-a_{1}^{2} w_{x}^{2}(x, t) +
		b_{1} w_{x}(x, t) w(x, t) + d_{1} w^{2}(x, t)\right) \mathrm{d}x \notag \\
		&\phantom{=}\;\; -a_{2}^{2} \int_{0}^{l} w_{x}(x, t - \tau) w_{x}(x, t) \mathrm{d}x + \label{EST1} \\
		&\phantom{=}\;\; \int_{0}^{l} \left(b_2 w_x (x, t - \tau ) + d_2 w(x,t - \tau)\right) w(x, t) \mathrm{d}x. \notag
	\end{align}
	Following the standard approach for delay differential equations (see, e.g., \cite{NiPi2008}),
	we define the history variable
	\begin{equation}
		z(x, t, s) := w(x, t - \tau s) \text{ for } (x, t, s) \in [0, l] \times [0, T] \times [0, 1]. \notag
	\end{equation}
	Exploiting the trivial equation
	\begin{equation}
		z_{t}(x, t, s) + \tau z_{s}(x, t, s) = 0 \text{ for } (x, t, s) \in (0, l) \times (0, T) \times (0, 1), \notag
	\end{equation}
	we arrive at the following distributional identity
	\begin{equation}
		z_{txx}(x, t, s) + \tau z_{sxx}(x, t, s) = 0 \text{ for } (x, t, s) \in (0, l) \times (0, T) \times (0, 1). \notag
	\end{equation}
	Multiplying this identity with $z(x, t, s)$, integrating over $(s, x) \in (0, 1) \times (0, l)$
	and carrying out a partial integration yields
	\begin{equation}
		\int_{0}^{1} \int_{0}^{l} \partial_{t} z^{2}_{x}(x, t, s) \mathrm{d}x \mathrm{d} s +
		\tau \int_{0}^{1} \partial_{s} z^{2}_{x}(x, t, s) \mathrm{d}x \mathrm{d}s = 0
		\text{ for } t \in (0, T). \notag
	\end{equation}
	Thus,
	\begin{equation}
		\partial_{t} \int_{0}^{1} \int_{0}^{l} z^{2}_{x}(x, t, s) \mathrm{d}x \mathrm{d}s +
		\tau \int_{0}^{l} z^{2}_{x}(x, t, s)\Big|_{s = 0}^{s = 1} \mathrm{d}x \text{ for } t \in (0, T), \notag
	\end{equation}
	i.e.,
	\begin{equation}
		\begin{split}
			\partial_{t} \int_{0}^{1} \int_{0}^{l} z^{2}_{x}(x, t, s) \mathrm{d}x \mathrm{d}s &=
			\tau \int_{0}^{1} w_{x}^{2}(x, t) \mathrm{d}x - \\
			&\phantom{=}\;\; \tau \int_{0}^{1} w_{x}^{2}(x, t - \tau) \mathrm{d}x \text{ for } t \in (0, T).
		\end{split}
		\label{EST2}
	\end{equation}
	Multiplying Equation (\ref{EST2}) with a constant $\omega > 0$ and adding the result to Equation (\ref{EST1}), we can estimate
	\begin{equation}
		\begin{split}
			\partial_{t} \int_{0}^{l} &w^{2}(x, t) \mathrm{d}x + \omega
			\partial_{t} \int_{0}^{1} \int_{0}^{l} w^{2}_{x}(x, t - \tau s) \mathrm{d}x \mathrm{d}s \leq \\
			&-\left(2 a_{1}^{2} - |b_{1}| \varepsilon - a_{2}^{2} \varepsilon\right) \int_{0}^{l} w_{x}^{2}(x, t) \mathrm{d}x \\
			&-\left(\omega \tau - \frac{a_{2}^{2}}{\varepsilon} - |b_{2}|\right) \int_{0}^{1} \int_{0}^{l} w^{2}_{x}(x, t - \tau s) \mathrm{d}x \mathrm{d}s \\
			&+ \left(2 d_{1} + |b_{2}| + |d_{2}| + \frac{|b_{1}|}{\varepsilon}\right) \int_{0}^{l} w^{2}(x, t) \mathrm{d}x.
		\end{split} 
		\notag
	\end{equation}
	Selecting now $\varepsilon > 0$ sufficiently small and $\omega > 0$ sufficiently large,
	we have shown
	\begin{equation}
		\begin{split}
			\partial_{t} &\left(\int_{0}^{l} w^{2}(x, t) \mathrm{d}x + \omega \int_{0}^{1} \int_{0}^{l} w^{2}_{x}(x, t - \tau s) \mathrm{d}x \mathrm{d}s\right) \leq \\
			&C \left(\int_{0}^{l} w^{2}(x, t) \mathrm{d}x + \omega \int_{0}^{1} \int_{0}^{l} w^{2}_{x}(x, t - \tau s) \mathrm{d}x \mathrm{d}s\right)
		\end{split}
		\notag
	\end{equation}
	for the constant $C := 2 d_{1} + |b_{2}| + |d_{2}| + \frac{|b_{1}|}{\varepsilon}$.
	From Gronwall's inequality we can thus conclude
	\begin{equation}
		\int_{0}^{l} w^{2}(x, t) \mathrm{d}x + \omega \int_{0}^{1} \int_{0}^{l} w^{2}_{x}(x, t - \tau s) \mathrm{d}x \mathrm{d}s \leq 0. \notag
	\end{equation}
	Therefore, $w \equiv 0$ implying $v_{1} \equiv v_{2}$.
\end{proof}

\begin{corollary}
	The solution $v$ depends continuously on the data $(g, \psi, \theta_{1}, \theta_{2})$ in the sense of the existence of a constant $C > 0$ such that
	\begin{equation}
		\begin{split}
			\int_{0}^{T} &\int_{0}^{l} \left(v^{2}(x, t) + \int_{0}^{1} v_{x}^{2}(x, t - \tau s) \mathrm{d}s\right) \mathrm{d}x \mathrm{d}t \leq C \bigg[\int_{0}^{T} \left(\int_{0}^{l} g^{2}(x, t) \mathrm{d}x \right. + \\
			& \dot{\theta}_{1}^{2}(t) + \dot{\theta}_{2}^{2}(t) \bigg) \mathrm{d}t +
			\int_{0}^{l} \psi^{2}(x, 0) \mathrm{d}x + \int_{0}^{1} \int_{0}^{l} \psi_{x}^{2}(x, t - \tau s) \mathrm{d}x \mathrm{d}s\bigg]
		\end{split}
		\notag
	\end{equation}
	for $g \in L^{2}\big((0, T), L^{2}\big((0, l)\big)\big)$,
	$\psi \in L^{2}\big((-\tau, 0), H^{1}\big((0, l)\big)\big)$ with $\varphi(0, \cdot) \in L^{2}\big((0, l)\big)$,
	$\theta_{1}, \theta_{2} \in H^{1}\big((0, T)\big)$.
\end{corollary}

In the following, we assume the coefficients $b_{1}$, $b_{2}$ at the first order derivatives to satisfy the following proportionality conditions
\begin{equation}
	-\frac{b_{1}}{2 a_{1}^{2}} = -\frac{b_{2}}{2 a_{2}^{2}} = \mu \notag
\end{equation}
for a certain $\mu \in \mathbb{R}$.
We substitute
\begin{equation}
	v(x, t) := e^{\mu x} u(x, t) \notag
\end{equation}
and obtain from equations (\ref{EQ2_1})--(\ref{EQ2_3}) an initial boundary value problem for the unknown function $u$
\begin{equation}
	u_{t}(x, t) = a_{1}^{2} u_{xx}(x, t) + a_{2}^{2} u_{xx}(x, t - \tau) + c_{1} u(x, t) + c_{2} u(x, t) + f(x, t)
	\label{NEW_DELAY_EQ_1}
\end{equation}
with
\begin{equation}
	c_{1} := d_{1} - \left(\frac{b_{1}}{2a_{1}}\right)^{2}, \quad
	c_{2} := d_{2} - \left(\frac{b_{2}}{2a_{2}}\right)^{2}, \quad
	f(x, t) := e^{-\mu x} g(x, t)
	\notag
\end{equation}
subject to the initial conditions
\begin{equation}
	u(x, t) = \varphi(x, t) \text{ for } x \in (0, l), t \in (-\tau, 0)
	\label{NEW_DELAY_EQ_2}
\end{equation}
with $\varphi(x, t) := e^{-\mu x} \psi(x, t)$, $x \in [0, l]$, $t \in [-\tau, 0]$,
and boundary conditions
\begin{equation}
	u(0, t) = \mu_{1}(t), \quad u(l, t) = \mu_{2}(t) \text{ for } t > -\tau
	\label{NEW_DELAY_EQ_3}
\end{equation}
with $\mu_{1}(t) := \theta_{1}(t)$, $\mu(t) := e^{-\mu l} \theta_{2}(t)$, $t \geq -\tau$.

Thus, there remains to establish the existence of a classical solution to (\ref{NEW_DELAY_EQ_1})--(\ref{NEW_DELAY_EQ_3})
which will be determined in the form
\begin{equation}
	u(x,t) = u_1 (x,t) + u_2 (x,t) + u_3 (x,t)
	\notag
\end{equation}
with the functions $u_1$, $u_2$, and $u_3$ given in what follows.

\begin{itemize}
	\item $u_1$ is the solution of the homogeneous equation
	\begin{equation}
		\begin{split}
			\frac{{\partial u_1 (x,t)}}{{\partial t}} &= a_1^2 \frac{{\partial ^2 u_1 (x,t)}}{{\partial x^2 }} + a_2^2 \frac{{\partial ^2 u_1 (x,t - \tau )}}{{\partial x^2 }} + \\
			&\phantom{=}\;\; c_1 u_1 (x,t) + c_2 u_1 (x,t - \tau)
		\end{split}
		\label{EQ2_4}
	\end{equation}
	subject to zero boundary conditions $u_{1}(0, t) = u_{1}(l, t) = 0$, $t > -\tau$, and non-zero initial conditions
	\begin{equation}
		u_{1}(x, t) = \Phi(x, t) \text{ for } x \in (0, l), t \in (-\tau, 0) \notag
	\end{equation}
	with
	\begin{equation}
		\Phi (x, t) := \varphi (x,t) - \mu _1 (t) - \frac{x}{l}\left[ {\mu _2 (t) - \mu _1 (t)} \right]
		\label{EQ2_5}
	\end{equation}
	for $x \in [0, l]$, $t \in [-\tau, 0]$.

	\item $u_2 (x,t)$ is the solution of the non-homogeneous equation
	\begin{equation}
		\begin{split}
			\frac{{\partial u_2 (x,t)}}{{\partial t}} &= a_1^2 \frac{{\partial ^2 u_2 (x,t)}}{{\partial x^2 }} + a_2^2 \frac{{\partial ^2 u_2 (x,t - \tau )}}{{\partial x^2 }} + \\
			&\phantom{=}\;\; c_1 u_2 (x,t) + c_2 u_2 (x,t - \tau ) + F(x,t)
		\end{split}
		\label{EQ2_6}
	\end{equation}
	with the right-hand side
	\begin{align}
		F(x, t) &:= f(x,t) - \frac{\mathrm{d}}{{\mathrm{d}t}}\left\{ {\mu _1 (t) + \frac{x}{l}\left[ {\mu _2 (t) - \mu _1 (t)} \right]} \right\} + \notag \\
		&\phantom{=}\;\; c_1 \left\{ {\mu _1 (t) + \frac{x}{l}\left[ {\mu _2 (t) - \mu _1 (t)} \right]} \right\} + \label{EQ2_7} \\
		&\phantom{=}\;\; c_2 \left\{ {\mu _1 (t - \tau ) + \frac{x}{l}\left[ {\mu _2 (t - \tau ) - \mu _1 (t - \tau )} \right]} \right\} \notag
	\end{align}
	subject to zero boundary conditions
	$u_{2}(0, t) = u_{2}(l, t) = 0$, $t > -\tau$, and
	zero initial conditions
	$u_{2}(x, t) = 0$, $x \in (0, l)$, $t \in (-\tau, 0)$.

	\item $u_3$ is the solution to the family of homogeneous elliptic equations
	\begin{equation}
		a_{1}^{2} \frac{{\partial^2 u_3 (x,t)}}{{\partial x^2}} = 0 \notag
	\end{equation}
	subject to non-zero boundary conditions
	$u_3(0, t) = \mu_{1}(t)$, $u_3(l, t) = \mu_{2}(t)$, $t > 0$.
	Thus, $u_{3}(x,t) = \mu _1 (t) + \frac{x}{l}\left[{\mu _2 (t) - \mu _1 (t)}\right]$.
\end{itemize}

\subsection{Homogeneous Equation with Delay}
\label{SUBSECTION_HOM_EQ_DELAY}
First, we consider homogeneous equation (\ref{EQ2_6}) with zero boundary and non-zero initial conditions.
The solution will be determined using Fourier's separation method.
Assuming
\begin{equation}
	u_1 (x,t) = X(x)T(t) \notag
\end{equation}
and plugging the ansatz into Equation (\ref{EQ2_6}), we get
\begin{equation}
	\begin{split}
		X(x)T'(t) &= a_1^2 X''(x)T(t) + a_2^2 X''(x)T(t - \tau ) + \\
		&\phantom{=}\;\; c_1 X(x)T(t) + c_2 X(x)T(t - \tau ).
	\end{split}
	\notag
\end{equation}
Collecting corresponding terms, we obtain
\begin{equation}
	\begin{split}
		[T'(t) - c_1 T(t) - c_2 T(t - \tau )]X(x) &= [a_1^2 T(t) + a_2^2 T(t - \tau )] X'(x).
	\end{split}
	\notag
\end{equation}
After separating the variables
\begin{equation}
	\frac{X''(x)}{X(x)} = \frac{{T'(t) - c_1 T(t) - c_2 T(t - \tau )}}{{a_1^2 T(t) + a_2^2 T(t - \tau )}} =  -\lambda^{2}, \notag
\end{equation}
the equation decomposes into two equations
\begin{align}
	X''(x) + \lambda^{2} X(x) &= 0, \label{EQ2_8} \\
	T'(t) - (c_1  - \lambda^{2} a_1^2 )T(t) + (c_2  - \lambda^{2} a_2^2) T(t - \tau ) &= 0. \label{EQ2_9}
\end{align}
Taking into account the boundary conditions for $u_{1}$, we get zero boundary conditions for $X$:
\begin{equation}
	X(0) = 0, \quad X(l) = 0. \notag
\end{equation}
Thus, nontrivial solutions of Equation (\ref{EQ2_8}) exist only for
\begin{equation}
	\lambda^{2}  = \lambda_n^{2}  = \left(\frac{\pi n}{l}\right)^{2}, \quad n \in \mathbb{N}.
	\label{EQ2_10}
\end{equation}
The latter are eigenvalues corresponding to the eigenfunctions
\begin{equation}
	X_n(x) = \sin \frac{\pi n}{l} x, \quad n \in \mathbb{N},
	\notag
\end{equation}
being solutions of the Sturm \& Liouville problem for the negative Dirichlet-Laplacian in $(0, l)$.
Plugging the values of $\lambda_n$, $n \in \mathbb{N}$, obtained in Equation (\ref{EQ2_10}) into Equation (\ref{EQ2_9}),
we obtain a countable system of decoupled ordinary delay differential equations
\begin{equation}
	\dot T_n(t) = \left[c_1 - \left(\frac{\pi n}{l}\right)^{2} a_{1}^{2}\right] T_n (t) + \left[c_2 - \left(\frac{\pi n}{l}\right)^{2} a_{2}^{2}\right] T_n (t - \tau ), \quad n \in \mathbb{N}.
	\label{EQ2_11}
\end{equation}
Similar to Section \ref{SECTION_HOM_EQUATION_NO_DELAY},
we consider the trivial extension of $(X_{n})_{n \in \mathbb{N}}$ to an orthogonal basis of $L^{2}\big((-\tau, 0), L^{2}\big((0, l)\big)\big)$.
Thus, if $\Phi \in L^{2}\big((-\tau, 0), L^{2}\big((0, l)\big)\big)$, we obtain for a.e. $(x, t) \in [0, l] \times [-\tau, 0]$
\begin{equation}
	\Phi(x, t) = \sum\limits_{k = 1}^\infty \Phi_n (t)  \sin \frac{\pi n}{l}x \text{ with }
	\Phi _n (t) = \frac{2}{l}\int\limits_0^l \Phi(\xi, t) \sin \frac{\pi n}{l}\xi \mathrm{d} \xi.
	\notag
\end{equation}
Taking into account Equation (\ref{EQ2_5}), we obtain initial conditions for the countably many ordinary delay differential equations (\ref{EQ2_9}) in the form
\begin{equation}
	T_n(t) = \Phi_n (t), \quad n \in \mathbb{N}, \quad t \in (-\tau, 0), \notag
\end{equation}
where
\begin{equation}
	\Phi_n (t) = \frac{2}{l}\int\limits_0^l \left\{ {\varphi (\xi ,t) - \left[ {\mu _1 (t)
	+ \frac{\xi }{l}\left[ {\mu _2 (t) - \mu _1 (t)} \right]} \right]} \right\} \sin \frac{\pi n}{l}\xi \mathrm{d} \xi.
	\notag
\end{equation}
These equations can be solved explicitly using well-known results on scalar linear ordinary delay differential equations (see, e.g., \cite{KhuShu2005}, \cite{KuKhu2011}).
In the following, we briefly outline this theory studying ordinary delay differential equations of the form
\begin{equation}
	\dot x(t) = ax(t) + bx(t - \tau) \text{ for } t \geq 0, \quad x(t) = \beta(t) \text{ for } t \in [-\tau, 0]
	\label{EQ2_12}
\end{equation}
where $\beta \in \mathcal{C}^{1}([-\tau, 0])$ is an arbitrary function representing the initial condition.

\begin{defin}
For $b \in \mathbb{R}$, $\tau > 0$,
the function $\mathbb{R} \ni t \mapsto \exp_\tau\{b, t\}$ given by
\begin{equation}
	\exp_\tau\{b, t\} :=
	\left\{
	\begin{array}{cc}
		0, & -\infty < t < -\tau, \\
		1, & -\tau \leq t < 0, \\
		1 + b \frac{t}{1!}, & 0 \leq t < \tau, \\
		1 + b \frac{t}{1!} + b^2 \frac{(t - \tau)^2}{2!}, & \tau \leq t < 2\tau,  \\
		\dots, & \dots \\
		1 + b \frac{t}{1!} + \cdots + b^k \frac{[t - (k - 1)\tau]^k}{k!}, & (k - 1) \tau \leq t < k\tau, k \in \mathbb{N}
	\end{array}\right.
	\notag
\end{equation}
is called the delayed exponential function.
\end{defin}

In \cite{KhuShu2005}, it has been proved that the delayed exponential function $\exp_{\tau}\{b, \cdot\}$ is the unique solution of the linear homogeneous equation with pure delay
\begin{equation}
	\dot x(t) = bx(t - \tau) \text{ for } t \geq 0
	\notag
\end{equation}
satisfying the identity initial condition $x(t) \equiv 1$ for $t \in [-\tau, 0]$.

It has further been shown that the solution to the general Cauchy problems for ordinary delay differential equations
also admit solutions of similar type. Namely, the following statements have been proved.
\begin{lemma}
	The function
	\begin{equation}
		x_0(t) = e^{at} \exp_\tau\{b_1, t\}, t \ge 0,
		\notag
	\end{equation}
	with $b_{1} := e^{-a \tau} b$
	is the unique solution of Equation (\ref{EQ2_12}) satisfying the initial condition
	\begin{equation}
		x_0 (t) = e^{at} \text{ for } t \in [-\tau, 0]. \notag
	\end{equation}
\end{lemma}
\begin{theorem}
	Equation (\ref{EQ2_12}) subject to general initial conditions
	$x(t) = \beta(t)$, $t \in [-\tau, 0]$, with $\beta \in \mathcal{C}^{1}\big([-\tau, 0]\big)$ is uniquely solved by a function
	$u \in \mathcal{C}^{0}\big([-\tau, \infty)\big) \cap \mathcal{C}^{1}\big([-\tau, 0]\big) \cap \mathcal{C}^{1}\big([0, \infty)\big)$ given via
	\begin{equation}
		\begin{split}
			x(t) &= e^{a(t + \tau )} \exp _\tau  \{ b_1 ,t\} \beta (-\tau ) + \\
			&\phantom{=}\;\; \int\limits_{ - \tau }^0 {e^{a(t - s)} \exp _\tau  \{ b_1 ,t - \tau  - s\} \left[ {\beta '(s) - a\beta (s)} \right] \mathrm{d}s}.
		\end{split}
		\label{EQ2_13}
	\end{equation}
\end{theorem}
\begin{remark}
	Using standard approximation arguments, the previous theorem can be easily generalized to the case $\beta \in W^{1, p}\big((-\tau, 0)\big)$, $p \in [1, \infty)$.
	The solution $u \in W^{1, p}_{\mathrm{loc}}\big((-\tau, \infty)\big)$
	satisfies then the equation (\ref{EQ2_12}) in distributional sense.
	The initial conditions can be interpreted in the sense of the continuous embedding
	$W^{1, p}\big((-\tau, \tau)\big) \hookrightarrow \mathcal{C}^{0}_{b}\big([-\tau, \tau]\big)$.
\end{remark}

We return now to Equation (\ref{EQ2_11}) with corresponding initial conditions.
Introducing the notation
\begin{equation}
	D_n = \left[c_2  - \left(\frac{\pi n}{l} a_{2}\right)^{2} \right] e^{-\left[c_1  - \left(\frac{\pi n}{l} a_{1}\right)^{2} \right] \tau}, \quad
	L_n = c_1 - \left(\frac{\pi n}{l} a_{1}\right)^{2}
	\label{EQ2_14}
\end{equation}
and using Equation (\ref{EQ2_13}), the solution to the problem (\ref{EQ2_11}) is given by
\begin{equation}
	\begin{split}
		T_n(t) &= e^{L_n (t + \tau)} \exp_\tau\{D_n, t\} \Phi_n(-\tau) + \\
		&\phantom{=}\;\; \int\limits_{ - \tau }^0 e^{L_n (t - s)} \exp _\tau  \{ D_n ,t - \tau  - s\} \left[ {\Phi '_n (s) - L_n \Phi _n (s)} \right] \mathrm{d}s.
	\end{split}
	\notag
\end{equation}
Thus, the solution to the initial boundary value problem (\ref{EQ2_4}) formally reads as
\begin{align}
	u_1 (x,t) &= \sum\limits_{n = 1}^\infty  \left\{ {e^{L_n (t + \tau )} \exp } \right._\tau \{ D_n ,t\} \Phi_n (-\tau) + \label{EQ2_15} \\
	&\phantom{=}\;\; \left. { + \int\limits_{-\tau }^0 {e^{L_n (t - s)} \exp\{D_n, t - \tau - s\} [\Phi '(s) - L_n \Phi (s)]ds} } \right\} \sin \frac{{\pi n}}{l}x \notag
\end{align}
with
\begin{equation}
	\Phi_n (t) = \frac{2}{l} \int\limits_0^l \left\{\varphi (\xi ,t) - \left[ {\mu _1 (t) + \frac{\xi }{l}[\mu _2 (t) - \mu _1 (t)]} \right] \right\} \sin\frac{{\pi n}}{l}\xi \mathrm{d} \xi.
	\label{EQ2_16}
\end{equation}
Conditions assuring the convergence of this Fourier series and the regularity of the limit function will be discussed later in this section.

\subsection{Non-Homogeneous Equation with Delay}
Next, we consider Equation (\ref{EQ2_6}), viz.,
\begin{equation}
	\begin{split}
		\frac{{\partial u_2 (x,t)}}{{\partial t}} &= a_1^2 \frac{{\partial ^2 u_2 (x,t)}}{{\partial x^2 }} + a_2^2 \frac{{\partial ^2 u_2 (x,t - \tau )}}{{\partial x^2 }} + \\
		&\phantom{=}\;\; c_1 u_2 (x,t) + c_2 u_2 (x,t - \tau ) + F(x, t)
	\end{split}
	\notag
\end{equation}
subject to zero boundary conditions
$u_{2}(0, t) = u_{2}(l, t) = 0$, $t > -\tau$, and zero initial conditions
$u_{2}(x, t) = 0$, $x \in (0, l)$, $t \in (-\tau, 0)$.
The solution will be obtained as a Fourier series with respect to the orthogonal eigenfunction basis of $L^{2}\big((-\tau, T), L^{2}\big((0, l)\big)\big)$, $T > 0$ arbitrary, but fixed
(cp. Section \ref{SECTION_NON_HOM_EQUATION_NO_DELAY}), i.e.,
\begin{equation}
	u_2(x,t) = \sum\limits_{n = 1}^\infty u_{2n} (t) \sin \frac{{\pi n}}{l}x, \quad n \in \mathbb{N}.
	\label{EQ2_17}
\end{equation}
Assuming $F \in L^{2}\big((0, T), L^{2}\big((0, l)\big)\big)$, the Fourier expansion of $F$ reads as
\begin{equation}
	F(x,t) = \sum\limits_{n = 1}^\infty  {F_n (t)e^{ - \frac{1}{2}\alpha x} \sin \frac{{\pi n}}{l}x} \text{ with }
	F_n (t) = \frac{2}{l}\int\limits_0^l F(s,t)e^{ - \frac{1}{2}\alpha \xi } \sin \frac{{\pi n}}{l}\xi \mathrm{d}\xi, \notag
\end{equation}
where $F$ is given in Equation (\ref{EQ2_7}) via
\begin{equation}
	\begin{split}
		F(x,t) &= f(x,t) - \frac{\mathrm{d}}{{\mathrm{d}t}}\left\{ {\mu _1 (t) + \frac{x}{l}[\mu _2 (t) - \mu _1 (t)]} \right\} + \\
		&\phantom{=}\;\; c_1 \left\{ {\mu _1 (t) + \frac{x}{l}[\mu _2 (t) - \mu _1 (t)]} \right\} + \\
		&\phantom{=}\;\; c_2 \left\{ {\mu _1 (t - \tau ) + \frac{x}{l}[\mu _2 (t - \tau ) - \mu _1 (t - \tau )]} \right\}.
	\end{split}
	\notag
\end{equation}

Then each of the functions $u_{2n}$, $n \in \mathbb{N}$, is the mild solution of the ordinary delay differential equation
\begin{equation}
	\dot u_{2n}(t) = \left[c_1  - \left(\frac{\pi n}{l} a_{1}\right)^{2} \right] u_{2n} (t) +
	\left[c_2  - \left(\frac{\pi n}{l} a_{2}\right)^{2} \right] u_{2n} (t - \tau ) + F_n (t).
	\notag
\end{equation}
Using the notation from Equation (\ref{EQ2_14}), the latter can be rewritten as
\begin{equation}
	\dot u_{2n} (t) = L_n u_{2n} (t) + D_n e^{L_n \tau} u_{2n} (t - \tau) + F_n(t)
	\label{EQ2_18}
\end{equation}
subject to zero initial conditions $u_{2n}(t) = 0$, $t \in (-\tau, 0)$.

Again, we present some auxiliary results from \cite{KhuShu2005}, \cite{KuKhu2011} for non-homogeneous ordinary delay differential equation of the form
\begin{equation}
	\dot x(t) = ax(t) + bx(t - \tau ) + \rho(t) \text{ for } t > 0
	\label{EQ2_19}
\end{equation}
with zero initial conditions $x(t) = 0$, $t \in (-\tau, 0)$.
\begin{theorem}
	Let $g \in \mathcal{C}^{0}\big([0, \infty)\big)$.
	The unique solution $\overline{x} \in \mathcal{C}^{0}\big([-\tau, \infty)\big) \cap \mathcal{C}^{1}\big([-\tau, 0]\big) \cap \mathcal{C}^{1}\big([0, \infty)\big)$ 
	of Equation (\ref{EQ2_19}) subject to zero initial conditions is given by
	\begin{equation}
		\overline{x}(t) =
		\left\{\begin{array}{cl}
			0, & t \in [-\tau, 0), \\
			\int\limits_0^t e^{a(t - s)} \exp _\tau \{b_1, t - \tau - s\} \rho(s) \mathrm{d}s, & t > 0
		\end{array}\right.
		\label{EQ2_20}
	\end{equation}
	where $b_1 := e^{-a\tau} b$.
\end{theorem}
\begin{remark}
	By exploiting standard approximation results,
	the previous theorem yields a unique mild solution
	$\overline{x} \in W^{1, p}_{\mathrm{loc}}\big((-\tau, \infty)\big)$, $p \in [1, \infty)$,
	for $\rho \in L^{p}_{\mathrm{loc}}\big((0, \infty)\big)$.
\end{remark}

Using Equation (\ref{EQ2_20}), the solution to the ordinary delay differential equation (\ref{EQ2_18}) subject to zero initial conditions can be written as
\begin{equation}
	u_{2n}(t) = \int\limits_0^t e^{L_n (t - s)} \exp_\tau\{D_n ,t - \tau - s\} F_n (s) \mathrm{d}s, \quad t \geq 0.
	\notag
\end{equation}
Hence, the solution to the non-homogeneous heat equation with delay (\ref{EQ2_7}) with zero boundary and initial conditions reads as
\begin{equation}
	u_2 (x,t) = \sum\limits_{n = 1}^\infty \left[ {\int\limits_0^t {e^{L_n (t - s)} \exp _\tau  \{ D_n, t - \tau - s\} } F_n (s)ds} \right]  \sin \frac{\pi n}{l} x.
	\label{EQ2_21}
\end{equation}
At the moment, Equation (\ref{EQ2_21}) gives only a formal representation formula.
Strict convergence conditions will though be given in the sequel.

Combing all relations from this section,
we obtain the solution to Equations (\ref{EQ2_1})--(\ref{EQ2_3}) in the form
\begin{align}
	u(x,t) &= \sum\limits_{n = 1}^\infty  \left\{ {e^{L_n (t + \tau )} \exp _\tau  \{ D_n ,t\} \Phi _n ( - \tau ) + } \right. \label{EQ2_22} \\
	&\phantom{=}\;\; \left. \int\limits_{ {-} \tau }^0 {e^{L_n (t {-} s)} \exp _\tau  \{ D_n ,t {-} \tau  {-} s\} [\Phi '_n (s) {-} L_n \Phi _n (s)]} \mathrm{d}s \right\} \sin \frac{{\pi n}}{l}x + \notag \\
	&\phantom{=}\;\; \sum\limits_{n = 1}^\infty  \left[ {\int\limits_0^t e^{L_n (t - s)} \exp _\tau  \{ D_n ,t - \tau  - s\} F_n (s) \mathrm{d} s} \right] \sin \frac{{\pi n}}{l}x + \notag  \\
	&\phantom{=}\;\; \mu _1(t) + \frac{x}{l} \left[\mu_2 (t) - \mu_1 (t)\right], \notag
\end{align}
where
\begin{align}
	D_n &= \left[c_2  - \left(\frac{\pi n}{l} a_{2}\right)^2 \right] e^{-\left[c_1  - \left(\frac{\pi n}{l} a_{1}\right)^2 \right] \tau}, \quad
	L_n = c_1 - \left(\frac{\pi n}{l} a_{1}\right)^2, \notag \\
	\Phi_n(t) &= \frac{2}{l}\int\limits_0^l \left\{\varphi (\xi ,t) - \left[{\mu _1 (t) + \frac{\xi }{l}\left[ {\mu _2 (t) - \mu _1 (t)} \right]} \right] \right\} \sin \frac{{\pi n}}{l}\xi \mathrm{d} \xi, \notag \\
	F_n (t) &= \frac{2}{l}\int\limits_0^l F(\xi ,t) \sin \frac{{\pi n}}{l}\xi \mathrm{d}\xi \text{ with } \label{EQ2_23} \\
	F(x, t) &= f(x, t) - \frac{\mathrm{d}}{\mathrm{d} t} \left[\mu_{1}(t) + \frac{x}{l} \left(\mu_{2}(t) - \mu_{1}(t)\right)\right] + \notag \\
	&\phantom{=}\;\; \left\{ {\mu _1 (t) {+} \frac{x}{l}[\mu _2 (t) {-} \mu _1 (t)]} \right\} {+} c_2 \left\{ {\mu _1 (t {-} \tau ) {+} \frac{x}{l}[\mu _2 (t {-} \tau ) {-} \mu _1 (t {-} \tau )]} \right\}. \notag
\end{align}

\subsection{Convergence of the Fourier Series}
Next, we discuss assumptions which assure the convergence of the Fourier series given in Equation (\ref{EQ2_22})
to the classical solution of the problem (\ref{EQ2_1})--(\ref{EQ2_3}).
We start with the following theorem giving rather technical conditions
which will later be interpreted in terms of Sobolev differentiability order.
\begin{theorem}
	\label{MAIN_THEOREM}
	Let $T > 0$ be fixed, $\delta > 0$ be arbitrary and
	let $m := \left\lceil\frac{T}{\tau}\right\rceil$.
	Further, let the functions $F$ and $\Phi$ defined from the data $f, \varphi, \mu_{1}, \mu_{2}$
	be such that that
	\begin{equation}
		F, F_{t} \in \mathcal{C}^{0}\big([0, l] \times [0, T]\big), \quad
		\Phi, \Phi_{t}, \Phi_{tt}, \Phi_{xx} \in \mathcal{C}^{0}\big([0, l] \times [-\tau, 0]\big) \notag
	\end{equation}
	and their Fourier coefficients $F_{n}$, $\Phi_{n}$, $n \in \mathbb{N}$,
	given in Equations (\ref{EQ2_16}), (\ref{EQ2_23}) satisfy the conditions
	\begin{equation}
		\begin{split}
			\lim\limits_{n \to \infty } n^{2m + 3 + \delta} |\Phi_n (-\tau)| &= 0, \\
			\lim\limits_{n \to \infty } n^{2m + 1 + \delta} 
			\max_{s \in [-\tau, 0]} \left[|\Phi''_{n}(s)| + n^{2} |\Phi'_{n}(s)| + n^{4} |\Phi_{n}(s)|\right] &= 0, \\
			\lim\limits_{n \to \infty } n^{2m - 1 + \delta}
			\max_{s \in [0, T]} \left[|F'_{n}(s)| + n^{2} |F_{n}(s)|\right] &= 0.
		\end{split}
		\label{EQ2_24}
	\end{equation}
	Then the Fourier series given in Equation (\ref{EQ2_22}) converges absolutely and uniformly with respect to $(x, t) \in [0, l] \times [0, T]$ to
	the classical solution $u$ of the problem (\ref{EQ2_1})--(\ref{EQ2_3}).
	Moreover, the Fourier series obtained by applying $\partial_{t}$, $\partial_{x}$ or $\partial_{xx}$ operators
	converge absolutely and unformly to $u_{t}$, $u_{x}$ or $u_{xx}$, respectively.
\end{theorem}

\begin{proof}
	We write the series from Equation (\ref{EQ2_22}) in the form
	\begin{equation}
		u(x, t) = S_1(x,t) + S_2 (x,t) + S_3(x,t) + \mu_1 (t) + \frac{x}{l}\left[ {\mu _2 (t) - \mu _1 (t)} \right],
		\notag
	\end{equation}
	where
	\begin{equation}
	\begin{split}
		S_1 (x,t) &= \sum\limits_{n = 1}^\infty {A_n (t) \sin \frac{{\pi n}}{l}x}, \quad
		S_2 (x,t) = \sum\limits_{n = 1}^\infty {B_n (t) \sin \frac{{\pi n}}{l}x}, \\
		S_3 (x,t) &= \sum\limits_{n = 1}^\infty {C_n } (t) \sin \frac{{\pi n}}{l}x, \quad
		A_n (t) = e^{L_n (t + \tau)} \exp_\tau\{D_n ,t\} \Phi_n(-\tau ), \\
		B_n (t) &= \int\limits_{ - \tau }^0 {e^{L_n (t - s)} \exp_\tau\{D_n ,t - \tau  - s\} [\Phi'_n (s) - L_n \Phi _n (s)]} \mathrm{d}s, \\
		C_n (t) &= \int\limits_0^t {e^{L_n (t - s)} \exp _\tau  \{ D_n ,t - \tau - s\}} F_n(s)\mathrm{d}s.
	\end{split}
	\notag
	\end{equation}
	\begin{enumerate}
		\item First, we consider the series $S_{1}$.
		For any fixed time $t^{\ast} \in [0, T]$ with $(k - 1) \tau \leq t^{\ast} < k\tau$, $k \leq m$, we get
		\begin{equation}
			\begin{split}
				A_n (t^\ast) &= e^{L_n (t^\ast + \tau)} \exp_\tau\{D_n, t^\ast\} \Phi_n (-\tau) \\
				&= e^{L_n (t^\ast  + \tau)} \left\{{1 + D_n \frac{{t^\ast}}{1!} + D_n^2 \frac{{(t^\ast  - \tau )^2 }}{2!} + \cdots + D_n^k \frac{{[t^\ast- (k - 1)\tau ]^k}}{{k!}}} \right\}\Phi_n (-\tau).
			\end{split}
			\notag
		\end{equation}
		Therefore, the series reads as
		\begin{equation}
			\begin{split}
				S_1(x,t^\ast) &= \sum\limits_{n = 1}^\infty  e^{L_n (t^\ast + \tau )} \Phi_n (-\tau) \sin \frac{{\pi n}}{l}x +  \\
				&\phantom{=}\;\; \frac{{t^\ast}}{{1!}}\sum\limits_{n = 1}^\infty  {e^{L_n (t^\ast  + \tau)} D_n \Phi _n (-\tau) \sin \frac{{\pi n}}{l}x + } \\
				&\phantom{=}\;\; \frac{{[t - (k - 2)\tau ]^{k - 1} }}{{(k - 1)!}}\sum\limits_{n = 1}^\infty  {e^{L_n (t^\ast + \tau)} D_n^{k - 1} \Phi_n(-\tau) \sin \frac{{\pi n}}{l}x + } \\
				&\phantom{=}\;\; \frac{{[t - (k - 1)\tau ]^k }}{{k!}}\sum\limits_{n = 1}^\infty  {e^{L_n (t^\ast + \tau )} D_n^k \Phi_n (-\tau) \sin \frac{{\pi n}}{l}x}.
			\end{split}
			\notag
		\end{equation}
		Plugging $L_{n}$ and $D_{n}$ from Equation (\ref{EQ2_14}) yields
		\begin{equation}
			\begin{split}
				S_1(x,t^\ast) &= \sum\limits_{n = 1}^\infty e^{\left[c_1 - \left(\frac{\pi n}{l} a_{1}\right)^{2} \right](t^\ast + \tau)} \Phi_n (-\tau) \sin \frac{\pi n}{l}x +  \\
				&\phantom{=}\;\;{+} \frac{{t^\ast}}{1!}\sum\limits_{n {=} 1}^\infty  e^{\left[c_1 - \left(\frac{\pi n}{l} a_{1}\right)^{2} \right] t^\ast} \left[c_2 - \left(\frac{\pi n}{l} a_{2}\right)^{2} \right] \Phi_n(-\tau) \frac{{\pi n}}{l}x + \dots \\
				&\phantom{=}\;\; \dots + \frac{{[t^\ast - (k - 2)\tau]^{k - 1} }}{{(k - 1)!}}\sum\limits_{n = 1}^\infty  {e^{\left[c_1 - \left(\frac{\pi n}{l} a_{1}\right)^{2} \right]\left[{t^\ast  - (k - 2)\tau}\right]}} \times \\
				&\phantom{=}\;\;\times \left[c_2 - \left(\frac{\pi n}{l} a_{2}\right)^{2} \right]^{k - 1} \Phi_n(-\tau) \sin \frac{{\pi n}}{l}x + \frac{{[t^\ast  - (k - 1)\tau ]^k}}{k!} \times \\
				&\phantom{=}\;\;\times \sum\limits_{n = 1}^\infty  {e^{\left[c_1 - \left(\frac{\pi n}{l} a_{1}\right)^{2} \right] \left[ {t^\ast - (k - 1)\tau} \right]} \left[c_2 - \left(\frac{\pi n}{l} a_{2}\right)^{2} \right]^k \Phi_n (-\tau)\sin \frac{{\pi n}}{l}x}.
			\end{split}
			\notag
		\end{equation}
		On the strength of condition $(k - 1) \tau \le t^\ast < k \tau$,
		for sufficiently large $n$, for which
		\begin{equation}
			c_1 - \left(\frac{\pi n}{l} a_{1}\right)^{2} < 0
			\notag
		\end{equation}
		holds true, the argument of the exponential function in the series becomes negative.
		Consider first $k$ terms.
		Since $t^\ast - j \tau > 0$, $j = - 1, \dots, k - 2$, these $k$ series converge absolutely and uniformly.

		Consider now the $(k+1)$-st term. The latter is given as a series of the following form
		\begin{equation}
			S_1^{k + 1}(x, t^\ast) = \frac{{[t^\ast - (k - 1)\tau ]^k }}{k!}\sum\limits_{n = 1}^\infty e^{\left[c_1 - \left(\frac{\pi n}{l} a_{1}\right)^{2} \right] \left[t^\ast   - (k - 1)\tau\right]} \times \notag
		\end{equation}
		\begin{equation}
			\times \left[c_2 - \left(\frac{\pi n}{l} a_{2}\right)^{2} \right]^k \Phi_n(-\tau) \sin \frac{{\pi n}}{l}x. \notag
		\end{equation}
		For $t^\ast \to (k - 1)\tau$, the argument of the exponential function approaches zero.
		Therefore, the series converges, but in general not uniformly with respect to $t^\ast \to (k - 1)\tau$.
		The latter is though maintained by our theorem assumptions
		since the Fourier coefficients $\Phi_n(-\tau)$ are assumed to be decaying sufficiently rapidly for $n \to \infty$, viz.,
		\begin{equation}
			\lim_{n \to \infty} n^{2k + 3 + \delta} |\Phi_n(-\tau)| \leq \lim_{n \to \infty} n^{2m + 3 + \delta} |\Phi_n (-\tau)| = 0.
			\notag
		\end{equation}
		Moreover, we can similarly conclude that the same holds for $\partial_{t} S_{1}$ and $\partial_{xx} S_{1}$
		since $\exp_{\tau}\{b, \cdot\}$ is continuously differentiable for $t \geq 0$
		and the application of $\partial_{t}$ and $\partial_{xx}$ operators corresponds, roughly speaking, to a term-wise multiplication with $n^{2}$.

		\item Next, we consider the second series $S_{2}$.
		For an arbitrary $t^{\ast} \in [0, T]$ with $(k - 1) \tau \leq t^{\ast} < k \tau$, $k \leq s$,
		we substitute $\xi := t^{\ast} - \tau - s$ and decompose the integral into two parts
		\begin{equation}
			\begin{split}
				B_{n}(t^{\ast}) &= \int\limits_{t^\ast - \tau}^{(k - 1)\tau} e^{L_n (\xi + \tau )} \exp_\tau\{D_n, \xi\} \left[ {\Phi '_n (t - \tau - \xi) - L_n \Phi _n (t - \tau  - \xi )} \right] \mathrm{d}\xi + \\
				&\phantom{=}\;\; \int\limits_{(k - 1)\tau }^{t^\ast} {e^{L_n (\xi + \tau )} \exp_\tau \{D_n, \xi \} \left[\Phi '_n (t - \tau - \xi) - L_n \Phi _n (t - \tau - \xi ) \right]} \mathrm{d}\xi.
			\end{split}
			\notag
		\end{equation}
		Exploiting the explicit form of the delayed exponential function on each of the time subintervals, we arrive at
		\begin{equation}
			\begin{split}
				B_n(t^\ast) &= \int\limits_{t^\ast  - \tau }^{(k - 1)\tau} {e^{L_n (\xi  + \tau )} [\Phi'_n (t^\ast - \tau - \xi) - L_n \Phi _n (t^\ast - \tau - \xi )]} \times \\
				&\phantom{=}\;\; \times \left\{ {1 + D_n \frac{\xi }{1!} + D_n^2 \frac{{(\xi  - \tau )^2 }}{{2!}} +  \cdots + D_n ^{k - 1} \frac{{[\xi  - (k - 2)\tau ]^{k - 1} }}{(k - 1)!}} \right\} \mathrm{d}\xi  + \\
				&\phantom{=}\;\; \int\limits_{(k - 1)\tau }^{t^\ast} {e^{L_n (\xi + \tau )} \left[ {\Phi '_n (t^\ast - \tau - \xi) - L_n \Phi _n (t^\ast - \tau  - \xi )} \right]}  \times \\
				&\phantom{=}\;\; \times \left\{1 + D_n \frac{\xi }{1!} + D_n^2 \frac{{(\xi  - \tau )^2 }}{{2!}} + \dots + D_n^k \frac{{[\xi  - (k - 1)\tau ]^k }}{k!} \right\} \mathrm{d}\xi.
			\end{split}
			\notag
		\end{equation}
		Plugging $L_{n}$ and $D_{n}$ from (\ref{EQ14}), we get
		\begin{equation}
			\begin{split} 
				B_n (t^*) &= \int\limits_{t^\ast - \tau }^{(k - 1)\tau } {e^{\left[c_{1} - \left(\frac{\pi n}{l} a_{1}\right)^{2}\right] (\xi  + \tau )} \left[\Phi'_n (t^\ast  - \tau - \xi ) - \right.} \\
				&\phantom{=}\;\; -\left[c_{1} - \left(\frac{\pi n}{l} a_{1}\right)^{2}\right]\left. {\Phi _n (t^\ast - \tau  - \xi )} \right] \times \\
				&\phantom{=}\;\; \times \left\{1 + \left[c_{2} - \left(\frac{\pi n}{l} a_{2}\right)^{2}\right] e^{-\left[c_{1} - \left(\frac{\pi n}{l} a_{1}\right)^{2}\right] \tau} \frac{\xi}{1!} + \right. \\
				&\phantom{=}\;\;+ \left[c_{2} - \left(\frac{\pi n}{l} a_{2}\right)^{2}\right]^2 e^{- 2\left[c_{1} - \left(\frac{\pi n}{l} a_{1}\right)^{2}\right] \tau} \frac{{(\xi  - \tau )^2}}{2!} + \ldots \\
				&\phantom{=}\;\;+ \left[c_{2} - \left(\frac{\pi n}{l} a_{2}\right)^{2}\right]^{k - 1} e^{-(k - 1)\left[c_{1} - \left(\frac{\pi n}{l} a_{1}\right)^{2}\right] \tau}  \times \\
				&\phantom{=}\;\;\times \left. {\frac{{[\xi  - (k - 2)\tau ]^{k - 1} }}{{(k - 1)!}}} \right\}\mathrm{d}\xi  + \int\limits_{(k - 1)\tau }^{t^\ast} {e^{\left[c_{1} - \left(\frac{\pi n}{l} a_{1}\right)^{2}\right] (\xi  + \tau )} }  \times \\
				&\phantom{=}\;\;\times \left\{ {\Phi'_n (t^\ast - \tau  - \xi ) - \left[c_{1} - \left(\frac{\pi n}{l} a_{1}\right)^{2}\right]} \right. \times \\
				&\phantom{=}\;\;\times e^{-\left[c_{1} - \left(\frac{\pi n}{l} a_{1}\right)^{2}\right] \tau } \frac{\xi }{{1!}} + \left[c_{2} - \left(\frac{\pi n}{l} a_{2}\right)^{2}\right]^2 e^{ {-} 2\left[c_{1} - \left(\frac{\pi n}{l} a_{1}\right)^{2}\right]}  \times \\
				&\phantom{=}\;\;\times \frac{{(\xi  - \tau )^2 }}{{2!}} +  \ldots  + \left[c_{2} - \left(\frac{\pi n}{l} a_{2}\right)^{2}\right]^k  \times \\
				&\phantom{=}\;\;\left. { \times e^{ - k\left[c_{1} - \left(\frac{\pi n}{l} a_{1}\right)^{2}\right]\tau } \frac{{[\xi  - (k - 1)\tau ]^k }}{{k!}}} \right\} \mathrm{d}\xi.
			\end{split}
			\notag
		\end{equation}
		Straightforward computations lead then to
		\begin{equation}
			\begin{split}
				B_n (t^\ast) &= \int\limits_{t^\ast - \tau }^{(k - 1)\tau } \left[\Phi'_n (t^\ast - \tau  - \xi ) {-} \left[c_{1} - \left(\frac{\pi n}{l} a_{1}\right)^{2}\right] \Phi _n (t^\ast - \tau - \xi ) \right] \times  \\
				&\phantom{=}\;\; \times \left\{ {e^{\left[c_{1} - \left(\frac{\pi n}{l} a_{1}\right)^{2}\right] (\xi  + \tau )}  + } \right.\left[c_{2} - \left(\frac{\pi n}{l} a_{2}\right)^{2}\right] \times \\
				&\phantom{=}\;\; \times e^{\left[c_{1} - \left(\frac{\pi n}{l} a_{1}\right)^{2}\right] \xi } \frac{\xi }{1!} + ... + \left[c_{2} - \left(\frac{\pi n}{l} a_{2}\right)^{2}\right]^{k - 1}  \times \\
				&\phantom{=}\;\; \times \left. {e^{\left[c_{2} - \left(\frac{\pi n}{l} a_{2}\right)^{2}\right] \left[\xi  - \left( {k - 2} \right)\tau\right] \hfill} \frac{{[\xi - (k - 2)\tau]^{k - 1} }}{{(k - 1)!}}} \right\} \mathrm{d} \xi  +
			\end{split}
			\notag
		\end{equation}
		\begin{equation}
			\begin{split}
				&\phantom{=}\;\; {+} \int\limits_{\left( {k - 1} \right)\tau }^{t^{\ast}} {\left[ {\Phi'_n (t^\ast  - \tau  - \xi) - \left[c_{1} - \left(\frac{\pi n}{l} a_{1}\right)^{2}\right] \Phi_n (t^\ast - \tau - \xi )} \right]}  \times \\
				&\phantom{=}\;\; \times \left\{ {e^{\left[c_{1} - \left(\frac{\pi n}{l} a_{1}\right)^{2}\right] (\xi  + \tau )}  + \left[c_{2} - \left(\frac{\pi n}{l} a_{2}\right)^{2}\right]} \right. \times \\
				&\phantom{=}\;\; \times e^{\left[c_{1} - \left(\frac{\pi n}{l} a_{1}\right)^{2}\right] \xi } \frac{\xi }{{1!}} + ... + \left[c_{2} - \left(\frac{\pi n}{l} a_{2}\right)^{2}\right]^k  \times \\
				&\phantom{=}\;\; \times \left. {e^{\left[c_{1} - \left(\frac{\pi n}{l} a_{1}\right)^{2}\right] [\xi - (k - 1)\tau ]} \frac{{[\xi - (k - 1)\tau ]^k }}{{k!}}} \right\} \mathrm{d}\xi.
			\end{split}
			\notag
		\end{equation}
		Thus, the $S_{2}(x, t^{\ast})$ can be written in the form
		\begin{equation}
			\begin{split}
				S_2&(x, t^\ast) = \sum\limits_{n = 1}^\infty {\left\{ {\int\limits_{t^\ast  - \tau }^{(k - 1)\tau } {e^{\left[c_{1} - \left(\frac{\pi n}{l} a_{1}\right)^{2}\right] \left( {\xi  + \tau } \right)} } } \right.}  \cdot \left[ {\Phi'_n (t^\ast   - \tau  - \xi) -} \right. \\
				&\phantom{=}\;\; \left. -\left. {\left[c_{1} - \left(\frac{\pi n}{l} a_{1}\right)^{2}\right] \Phi _n (t^\ast   - \tau  - \xi )} \right]\mathrm{d}\xi \right\} \sin \frac{{\pi n}}{l}x + \\
				&\phantom{=}\;\; + \sum\limits_{n = 1}^\infty  {\left\{ {\int\limits_{t^\ast  - \tau }^{(k - 1)\tau } {e^{\left[c_{1} - \left(\frac{\pi n}{l} a_{1}\right)^{2}\right] \xi } } } \right.} \frac{\xi }{1!}\left[c_{2} - \left(\frac{\pi n}{l} a_{2}\right)^{2}\right] \times \\
				&\phantom{=}\;\; \times \left. {\left[{\Phi'_n (t^\ast  - \tau  - \xi ) - \left[c_{1} - \left(\frac{\pi n}{l} a_{1}\right)^{2}\right] \Phi _n (t^\ast  - \tau - \xi )} \right]\mathrm{d}\xi } \right\} \times \\
				&\phantom{=}\;\; \times \sin \frac{{\pi n}}{l}x + \dots + \sum\limits_{n {=} 1}^\infty  {\left\{ {\int\limits_{t^\ast - \tau }^{(k - 1)\tau } {e^{\left[c_{1} - \left(\frac{\pi n}{l} a_{1}\right)^{2}\right]} } } \right.} \frac{{\left[ {\xi - (k - 2)\tau } \right]^{k {-} 1} }}{{(k - 1)!}} \times \\
				&\phantom{=}\;\; \times \left[c_{2} - \left(\frac{\pi n}{l} a_{2}\right)^{2}\right]^{k - 1} \times \left[ {\Phi'_n (t^\ast - \tau  {-} \xi ) - \left[c_{1} - \left(\frac{\pi n}{l} a_{1}\right)^{2}\right] \times } \right. \\
				&\phantom{=}\;\; \times \left. {\left. {\Phi_n (t^\ast   {-} \tau  - \xi )} \right]\mathrm{d}\xi } \right\} \sin \frac{{\pi n}}{l}x + \sum\limits_{n = 1}^\infty  {\left\{ {\int\limits_{(k - 1)\tau }^{t^\ast} {e^{\left[c_{1} - \left(\frac{\pi n}{l} a_{1}\right)^{2}\right] (\xi + \tau )} } } \right.}  \times \\
				&\phantom{=}\;\; \times \left. {\left[ {\Phi'_n (t^\ast   {-} \tau  - \xi ) - \left[c_{1} - \left(\frac{\pi n}{l} a_{1}\right)^{2}\right] \Phi _n (t^\ast - \tau  - \xi )} \right]\mathrm{d}\xi } \right\} \times \\
				&\phantom{=}\;\; \times \sin \frac{{\pi n}}{l}x + \sum\limits_{n = 1}^\infty  {\left\{ {\int\limits_{(k {-} 1)\varsigma }^{t^\ast} {e^{\left[c_{1} - \left(\frac{\pi n}{l} a_{1}\right)^{2}\right] \xi } } \frac{\xi }{{1!}} \left[c_{2} - \left(\frac{\pi n}{l} a_{2}\right)^{2}\right] \times } \right.}
			\end{split}
			\notag
		\end{equation}
		\begin{equation}
			\begin{split}
				&\phantom{=}\;\; \times \left. {\left[ {\Phi'_n (t^\ast   - \tau  - \xi ) - \left[c_{1} - \left(\frac{\pi n}{l} a_{1}\right)^{2}\right] \Phi _n (t^\ast - \tau  - \xi )} \right]\mathrm{d}\xi } \right\} \times \\
				&\phantom{=}\;\; \times \sin \frac{{\pi n}}{l}x + \dots + \sum\limits_{n = 1}^\infty  {\left\{ {\int\limits_{(k - 1)\tau }^{t^\ast} {e^{\left[c_{1} - \left(\frac{\pi n}{l} a_{1}\right)^{2}\right] \left[ {\xi  - (k - 2)\tau } \right]} } } \right.}  \times \\
				&\phantom{=}\;\; \times \frac{{\left[ {\xi  - (k - 2)\tau } \right]^{k - 1} }}{{(k - 1)!}}\left[c_{2} - \left(\frac{\pi n}{l} a_{2}\right)^{2}\right]^{k - 1}  \times \\
				&\phantom{=}\;\; \times \left. {\left[ {\Phi'_n (t^\ast   - \tau  - \xi ) - \left[c_{1} - \left(\frac{\pi n}{l} a_{1}\right)^{2}\right]\Phi _n (t^\ast - \tau  - \xi )} \right]\mathrm{d}\xi } \right\} \times \\
				&\phantom{=}\;\; \times \sin \frac{{\pi n}}{l}x + \sum\limits_{n {=} 1}^\infty  {\left\{ {\int\limits_{(k {-} 1)\tau }^{t^\ast} {e^{\left[c_{1} - \left(\frac{\pi n}{l} a_{1}\right)^{2}\right] [\xi  {-} (k {-} 1)\tau ]} } } \right.} \frac{{[\xi  {-} (k {-} 1)\tau ]^k }}{{k!}} \times \\
				&\phantom{=}\;\; \times \left[c_{2} - \left(\frac{\pi n}{l} a_{2}\right)^{2}\right]^k \bigg[\Phi '_n (t^\ast  - \tau  - \xi ) - \\
				&\phantom{=}\;\; \left. {- \left. {\left[c_{1} - \left(\frac{\pi n}{l} a_{1}\right)^{2}\right] \Phi _n (t^ *   - \tau  - \xi )\mathrm{d}\xi } \right]} \right\} \sin \frac{{\pi n}}{l}x.
			\end{split}
			\notag
		\end{equation}
		On the strength of condition $(k - 1)\tau \leq t^{\ast} < k\tau$, we have for sufficiently large $n$
		\begin{equation}
			c_{1} - \left(\frac{\pi n}{l} a_{1}\right)^{2}  < 0, \notag
		\end{equation}
		which guarantees the arguments of the exponential function to be negative.
		Hence, all but the last series converge absolutely and uniformly.
		We consider thus the last series
		\begin{equation}
			\begin{split}
				&S_2^{k + 1}(x, t^{\ast}) = \sum\limits_{n = 1}^\infty \left\{\int\limits_{(k - 1)\tau}^{t^{\ast}} e^{\left[c_{1} - \left(\frac{\pi n}{l} a_{1}\right)^{2}\right] \left[\xi  - (k - 1)\tau\right]} \times \right. \\
				&\phantom{=}\;\; \times \frac{[\xi  - (k - 1)\tau ]^k}{k!}\left[c_{2} - \left(\frac{\pi n}{l} a_{2}\right)^{2}\right]^k  \times \\
				&\phantom{=}\;\; \times \left.\left[\Phi'_n (t^{\ast} - \tau - \xi) - \left[c_{1} - \left(\frac{\pi n}{l} a_{1}\right)^{2}\right] \Phi_n (t^{\ast}  - \tau - \xi) \right]\mathrm{d}\xi \right\} \sin \frac{{\pi n}}{l}x.
			\end{split}
			\notag
		\end{equation}
		Exploiting the assumption that the Fourier coefficients $\Phi_{n}(t^{\ast})$, $\Phi_{n}'(t^{\ast})$ are rapidly decreasing for $n \to \infty$, i.e.,
		\begin{equation}
			\lim\limits_{n \to \infty } n^{2m + 1 + \delta} 
			\max_{s \in [-\tau, 0]} \left[|\Phi''_{n}(s)| + n^{2} |\Phi'_{n}(s)| + n^{4} |\Phi_{n}(s)|\right] = 0 
			\notag
		\end{equation}
		for some $\delta > 0$ and taking into account the smoothness of $\exp_{\tau}\{b, \cdot\}$ function,
		we conclude like in the previous case that the whole series $S_{2}$ as well as $\partial_{t} S_{2}$, $\partial_{xx} S_{2}$ converge absolutely and uniformly.

		\item Finally, we consider the third series $S_{3}$.
		For an arbitrary $t^{\ast}$ with $(k - 1) \tau \leq t^{\ast} < k \tau$, $k \leq s$,
		we substitute $\xi := t^{\ast} - \tau - s$ and rewrite the integral as a Fourier series with the following integral coefficients
		\begin{equation}
			\begin{split}
				C_n (t^{\ast}) &= \int\limits_{-\tau}^{t^{\ast} - \tau} e^{\left[c_{1} - \left(\frac{\pi n}{l} a_{1}\right)^{2}\right] (\xi  + \tau )} \exp_\tau\{D_n, \xi\} F_n(t^{\ast} - \tau - \xi) \mathrm{d}\xi \\
				&= \int\limits_{-\tau}^0 e^{\left[c_{1} - \left(\frac{\pi n}{l} a_{1}\right)^{2}\right](\xi  + \tau)} F_n (t^{\ast} - \tau - \xi) \mathrm{d}\xi + \\
				&\phantom{=}\;\; \int\limits_0^\tau  {e^{\left[c_{1} - \left(\frac{\pi n}{l} a_{1}\right)^{2}\right]](\xi  + \tau )} \left[ {1 + D_n \frac{\xi }{{1!}}} \right]} F_n (t^\ast - \tau  - \xi ) \mathrm{d}\xi  + \dots \\
				&\phantom{=}\;\; \dots + \int\limits_{(k - 2)\tau }^{t^{\ast} - \tau} e^{\left[c_{1} - \left(\frac{\pi n}{l} a_{1}\right)^{2}\right] (\xi  + \tau)} \left[1 + D_n \frac{\xi}{1!} + D_n^2 \frac{(\xi - \tau)^2}{2!} + \dots \right. \\
				&\phantom{=}\;\; \dots \left. + D_n ^{k - 1} \frac{(\xi  - (k - 2)\tau)^{k - 1}}{(k - 1)!}\right] F_n (t^{\ast} - \tau - \xi ) \mathrm{d}\xi.
			\end{split}
			\notag
		\end{equation}
		Hence, $S_{3}$ reads as
		\begin{equation}
			\begin{split}
				S_3&(x,t^{\ast}) = \sum\limits_{n = 1}^\infty \left\{\int\limits_{-\tau}^0 {e^{\left[c_{1} - \left(\frac{\pi n}{l} a_{1}\right)^{2}\right] (\xi + \tau)} F_n (t^{\ast} - \tau - \xi )} \mathrm{d}\xi \right\} \sin \frac{\pi n}{l}x + \\
				&\phantom{=}\;\; \sum\limits_{n = 1}^\infty \left\{\int\limits_0^\tau \left[e^{\left[c_{1} - \left(\frac{\pi n}{l} a_{1}\right)^{2}\right] (\xi + \tau)}  + \left[c_{2} - \left(\frac{\pi n}{l} a_{2}\right)^{2}\right] \right. \times \right. \\
				&\phantom{=}\;\; \left. \times \left. e^{\left[c_{1} - \left(\frac{\pi n}{l} a_{1}\right)^{2}\right]\xi} \frac{\xi }{1!}\right] F_n (t^{\ast}  - \tau - \xi) \mathrm{d}\xi \right\} \sin \frac{\pi n}{l}x + \dots \\
				&\phantom{=}\;\; \dots + \sum\limits_{n = 1}^\infty {\left\{ {\int\limits_{(k - 2)\tau }^{t^{\ast} - \tau} {\left[ {e^{\left[c_{1} - \left(\frac{\pi n}{l} a_{1}\right)^{2}\right] (\xi  + \tau )}  + } \right.} } \right.} \left[c_{2} - \left(\frac{\pi n}{l} a_{2}\right)^{2}\right] \times \\
				&\phantom{=}\;\; \times e^{\left[c_{1} - \left(\frac{\pi n}{l} a_{1}\right)^{2}\right] \xi } \frac{\xi }{{1!}} + \dots + \left[c_{2} - \left(\frac{\pi n}{l} a_{2}\right)^{2}\right]^{k - 1}  \times
			\end{split}
			\notag
		\end{equation}
		\begin{equation}
			\left.{\times \frac{{[t^{\ast}   - (k - 2)\tau ]^{k - 1} }}{{(k - 1)!}}e^{\left[c_{1} - \left(\frac{\pi n}{l} a_{1}\right)^{2}\right] [\xi  - (k - 2)\tau ]} } \right] \times
			\times F_n (t^{\ast} - \tau  - \xi) \mathrm{d}\xi \Bigg\} \sin \frac{{\pi n}}{l}x.
			\notag
		\end{equation}
		Similarly to the previous cases,
		for sufficiently large $n$ we have
		\begin{equation}
			c_{1} - \left(\frac{\pi n}{l} a_{1}\right)^{2} < 0, \notag
		\end{equation}
		which yields the convergence of all but the first two subseries in each of the series due to the negativity of the argument of the exponential function.
		We consider a typical term where the convergence is not maintained by the exponential function
		\begin{equation}
			\begin{split}
				S_3^k(x, t^{\ast}) &= \sum\limits_{n = 1}^\infty \bigg\{\int\limits_{(k - 2) \tau}^{t^{\ast} - \tau}
				\left[c_{2} - \left(\frac{\pi n}{l}\right)^{2}\right]^{k-1}
				\frac{\left[t^{\ast} - (k - 2)\tau\right]}{(k - 1)!} \times \\
				&\phantom{=} \times e^{\left[c_1  - \left(\frac{\alpha^2}{4} + \frac{\pi ^2 n^2}{l^2}\right) a_1^2\right](\xi - (k - 2))}  F_n(t^{\ast} - \tau - \xi) \mathrm{d}\xi\bigg\} \sin \frac{\pi n}{l}x.
			\end{split}
			\notag
		\end{equation}	
		Exploiting the assumption
		\begin{equation}
			\lim\limits_{n \to \infty } n^{2m - 1 + \delta}
			\max_{s \in [0, T]} \left[|F'_{n}(s)| + n^{2} |F_{n}(s)|\right] = 0 \notag
		\end{equation}
		as well as the fact $\exp_{\tau}\{b, \cdot\}$ is Lipschitz-continuous on $[-\tau, \infty)$,
		we similarly conclude the absolute and uniform converge of $S_{3}$, $\partial_{t} S_{3}$, and $\partial_{xx} S_{3}$.
	\end{enumerate}
	This ends the proof.
\end{proof}

Finally, we give a Sobolev space interpretation of the conditions of Theorem \ref{MAIN_THEOREM}.
As a direct consequence of Lemma \ref{LEMMA_FOURIER_DECAY},
we know that the conditions of Theorem \ref{MAIN_THEOREM} are fulfilled if the functions $F$ and $\Phi$ from Equations (\ref{EQ2_5}) and (\ref{EQ2_7}) satisfy
\begin{equation}
	\begin{split}
		\Phi &\in \mathcal{C}^{2}\big([-\tau, 0], X_{m + 1}\big) \cap \mathcal{C}^{1}\big([-\tau, 0], X_{m + 2}\big) \cap \mathcal{C}^{0}\big([-\tau, 0], X_{m + 3}\big), \\
		F &\in \mathcal{C}^{1}\big([0, T], X_{m}\big) \cap \mathcal{C}^{0}([0, T], X_{m + 1}\big).
	\end{split}
	\notag
\end{equation}
Taking into account Definition \ref{DEFINITION_OPERATOR} and Equations (\ref{EQ2_5}) and (\ref{EQ2_7})
and using the explicit characterization $X_{m}$, the latter holds under the assumptions of the following Corollary.
\begin{corollary}
	The decay conditions on the Fourier coefficients in the previous theorem are satisfied if the following regularity
	\begin{equation}
		\begin{split}
			\varphi &\in \bigcap_{k = 0}^{2} \mathcal{C}^{k}\big([-\tau, 0], H^{2m + 2(3 - k)}\big((0, l)\big)\big), \\
			\mu_{1}, \mu_{2} &\in \mathcal{C}^{3}([0, T]), \\
			f &\in \mathcal{C}^{1}\big([0, T], H^{2m}\big((0, l)\big)\big) \cap \mathcal{C}^{0}\big([0, T], H^{2m + 2}\big((0, l)\big)\big)
		\end{split}
		\notag
	\end{equation}
	and compatibility conditions
	\begin{equation}
		\begin{split}
			&\varphi(0, t) = \mu_{1}(t), \quad \varphi(l, t) = \mu_{2}(t), \\
			&\partial_{x}^{2j} \partial_{t}^{k} \varphi(0, t) = 0, \quad
			\partial_{x}^{2j} \partial_{t}^{k} \varphi(l, t) = 0, \quad k = 0, 1, 2, \; j = 1, \dots, m + 2 - k
		\end{split}
		\notag
	\end{equation}
	for $t \in [-\tau, 0]$ and
	\begin{equation}
		\begin{split}
			f(0, t) - \dot{\mu}_{1}(t) + c_{1} \mu_{1}(t) + c_{2} \mu_{1}(t - \tau) &= 0, \\
			f(l, t) - \dot{\mu}_{2}(t) + c_{1} \mu_{2}(t) + c_{2} \mu_{2}(t - \tau) &= 0, \\
			f_{t}(0, t) - \ddot{\mu}_{1}(t) + c_{1} \dot{\mu}_{1}(t) + c_{2} \dot{\mu}_{1}(t - \tau) &= 0, \\
			f_{t}(l, t) - \ddot{\mu}_{2}(t) + c_{1} \dot{\mu}_{2}(t) + c_{2} \dot{\mu}_{2}(t - \tau) &= 0, \\
			\partial_{x}^{2j} f(0, t) = 0, \quad \partial_{x}^{2j} f(l, t) &= 0, \quad j = 0, \dots, m \\
			\partial_{x}^{2j} \partial_{t} f(0, t) = 0, \quad \partial_{x}^{2j} \partial_{t} f(l, t) &= 0, \quad j = 0, \dots, m - 1
		\end{split}
		\notag
	\end{equation}
	for $T \in [0, T]$ are satisfied.
\end{corollary}

\AuthorInfo{Denis Khusainov, \\Taras Shevchenko National University of Kyiv, Kyiv, 03680, Ukraine;}

\AuthorInfo{Michael Pokojovy, \\ University of Konstanz, Konstanz, 78457, Germany;}

\AuthorInfo{Elvin Azizbayov, \\ Baku State University, Baku, AZ 1148, Azerbaijan.}
\end{document}